\documentclass[fleqn,11pt]{elsarticle}  
\usepackage{amsmath}
\usepackage{amssymb}
\usepackage{theorem}
\usepackage{euscript}
\usepackage{enumitem}
\usepackage{pstricks}
\title{Proximity for Sums of Composite Functions\tnoteref{t1}}
\tnotetext[t1]{The work of P. L. Combettes was supported the 
Agence Nationale de la Recherche under grant ANR-08-BLAN-0294-02. 
The work of D\hspace{-1.6ex}\raise 0.4ex\hbox{-}\hspace{.8ex}inh D\~ung 
and B$\grave{\text{\u{a}}}$ng C\^ong V\~u was supported by the 
Vietnam National Foundation for Science and Technology Development.}
\author[plc]{Patrick L. Combettes\corref{cor1}}
\ead{plc@math.jussieu.fr}
\cortext[cor1]{Corresponding author}
\author[dd]{D\hspace{-1.6ex}\raise 0.4ex\hbox{-}\hspace{.8ex}inh D\~ung}
\ead{dinhdung@vnu.edu.vn}
\author[bcv]{B$\grave{\text{\u{a}}}$ng C\^ong V\~u}
\ead{vu@ann.jussieu.fr}
\address[plc]{UPMC Universit\'e Paris 06, 
Laboratoire Jacques-Louis Lions -- UMR 7598,
75005 Paris, France.\\
Tel.: +33 1 4427 6319,
Fax:  +33 1 4427 7200,
plc@math.jussieu.fr}
\address[dd]{Vietnam National University, 
Information Technology Institute, Hanoi, Vietnam}
\address[bcv]{UPMC Universit\'e Paris 06, 
Laboratoire Jacques-Louis Lions -- UMR 7598,
75005 Paris, France.}
\newcommand{\scal}[2]{{\left\langle{{#1}\mid{#2}}\right\rangle}}
\newcommand{\menge}[2]{\big\{{#1}~\big |~{#2}\big\}}

\newcommand{\BL}{\ensuremath{\EuScript B}\,}
\newcommand{\HHH}{\ensuremath{\boldsymbol{\mathcal H}}}
\newcommand{\GGG}{\ensuremath{\boldsymbol{\mathcal G}}}
\newcommand{\RX}{\ensuremath{\left]-\infty,+\infty\right]}}

\newcommand{\HH}{\ensuremath{{\mathcal H}}}

\newcommand{\GG}{\ensuremath{{\mathcal G}}}

\newcommand{\emp}{\ensuremath{{\varnothing}}}
\newcommand{\prox}{\ensuremath{\operatorname{prox}}}

\newcommand{\Id}{\ensuremath{\operatorname{Id}}\,}

\newcommand{\cone}{\ensuremath{\operatorname{cone}}}
\newcommand{\RR}{\ensuremath{\mathbb{R}}}

\newcommand{\RPP}{\ensuremath{\left]0,+\infty\right[}}
\newcommand{\NN}{\ensuremath{\mathbb N}}

\newcommand{\cart}{\ensuremath{\raisebox{-0.5mm}{\mbox{\LARGE{$\times$}}}}\!}

\newcommand{\pinf}{\ensuremath{{+\infty}}}

\newcommand{\dom}{\ensuremath{\operatorname{dom}}}
\newcommand{\sri}{\ensuremath{\operatorname{sri}}}

\newcommand{\reli}{\ensuremath{\operatorname{ri}}}
\newcommand{\spc}{\ensuremath{\overline{\operatorname{span}}\,}}
\newcommand{\spa}{\ensuremath{{\operatorname{span}}\,}}

\newcommand{\inte}{\ensuremath{\operatorname{int}}}

\newtheorem{theorem}{Theorem}[section]
\newtheorem{corollary}[theorem]{Corollary}
\newtheorem{problem}[theorem]{Problem}
\theoremstyle{plain}{\theorembodyfont{\rmfamily}%
\newtheorem{algorithm}[theorem]{Algorithm}}
\theoremstyle{plain}{\theorembodyfont{\rmfamily}%
\newtheorem{remark}[theorem]{Remark}}
\numberwithin{equation}{section}
\begin{document}

\begin{abstract}
We propose an algorithm for computing the proximity operator of a
sum of composite convex functions in Hilbert spaces and investigate 
its asymptotic behavior. Applications to best approximation and image
recovery are described.
\end{abstract}

\begin{keyword}
Best approximation \sep 
convex optimization \sep 
duality \sep 
image recovery \sep 
proximity operator \sep 
proximal splitting algorithm
\end{keyword}
\maketitle

\section{Introduction}
Let $\HH$ be a real Hilbert space with scalar product 
$\scal{\cdot}{\cdot}$ and associated norm $\|\cdot\|$. The best 
approximation to a point $z\in\HH$ from a nonempty closed convex 
set $C\subset\HH$ is the point $P_Cz\in C$ that satisfies 
$\|P_Cz-z\|=\min_{x\in C}\|x-z\|$. The induced best approximation
operator $P_C\colon\HH\to C$, also called the projector onto $C$,
plays a central role in several branches of applied 
mathematics \cite{Deut01}. 
If we designate by $\iota_C$ the indicator function of $C$, i.e.,
\begin{equation}
\label{e:iota}
\iota_C\colon x\mapsto
\begin{cases}
0,&\text{if}\;\;x\in C;\\
\pinf,&\text{if}\;\;x\notin C,
\end{cases}
\end{equation}
then $P_Cz$ is the solution to the minimization problem 
\begin{equation}
\label{e:proj}
\underset{x\in\HH}{\mathrm{minimize}}\;\;\iota_C(x)+\frac12\|x-z\|^2.
\end{equation}
Now let $\Gamma_0(\HH)$ be the class of lower semicontinuous convex 
functions $f\colon\HH\to\RX$ such that
$\dom f=\menge{x\in\HH}{f(x)<\pinf}\neq\emp$. In \cite{Mor62b}
Moreau observed that, for every function $f\in\Gamma_0(\HH)$, 
the proximal minimization problem 
\begin{equation}
\label{e:prox}
\underset{x\in\HH}{\mathrm{minimize}}\;\;f(x)+\frac12\|x-z\|^2
\end{equation}
possesses a unique solution, which he denoted by $\prox_fz$.
The resulting proximity operator $\prox_f\colon\HH\to\HH$ 
therefore extends the notion of a best approximation operator
for a convex set. This fruitful concept has become a central tool
in mechanics, variational analysis, optimization, and signal 
processing, e.g., \cite{Alar06,Banf10,Rock09}.

Though in certain simple cases closed-form expressions are available
\cite{Banf10,Smms05,More65}, computing $\prox_fz$ in numerical
applications is a challenging task. The objective of this paper is 
to propose a splitting algorithm to compute proximity operators in 
the case when $f$ can be decomposed as a sum of composite functions. 

\begin{problem}
\label{prob:1}
Let $z\in\HH$ and let $(\omega_i)_{1\leq i\leq m}$ be reals in 
$\left]0,1\right]$ such that $\sum_{i=1}^m\omega_i=1$. 
For every $i\in\{1,\ldots,m\}$, let $(\GG_i,\|\cdot\|_{\GG_i})$ be 
a real Hilbert space, let $r_i\in\GG_i$, let $g_i\in\Gamma_0(\GG_i)$,
and let $L_i\colon\HH\to\GG_i$ be a nonzero bounded linear operator.
The problem is to 
\begin{equation}
\label{e:prob1}
\underset{x\in\HH}{\mathrm{minimize}}\;\;
\sum_{i=1}^m\omega_ig_i(L_ix-r_i)+\frac12\|x-z\|^2.
\end{equation}
\end{problem}

The underlying practical assumption we make is that the proximity
operators $(\prox_{g_i})_{1\leq i\leq m}$ are implementable 
(to within some quantifiable error). We are therefore aiming
at devising an algorithm that uses these operators separately.
Let us note that such splitting algorithms are already available 
to solve Problem~\ref{prob:1} under certain restrictions.
\begin{enumerate}
\item[A)]
Suppose that $\GG_1=\HH$, that $L_1=\Id$, that the functions 
$(g_i)_{2\leq i\leq m}$ are differentiable everywhere with a 
Lipschitz continuous gradient, and that $r_i\equiv 0$. Then 
\eqref{e:prob1} reduces to the minimization of the sum of 
$f_1=g_1\in\Gamma_0(\HH)$ and of the smooth function 
$f_2=\sum_{i=2}^m\omega_ig_i\circ L_i+\|\cdot-z\|^2/2$,
and it can be solved by the forward-backward 
algorithm \cite{Smms05,Tsen91}.
\item[B)]
The methods proposed in \cite{Joca09} address the case
when, for every $i\in\{1,\ldots,m\}$, $\GG_i=\HH$, $L_i=\Id$, 
and $r_i=0$. 
\item[C)]
The method proposed in \cite{Svva10} addresses the case
when $m=2$, $\GG_1=\HH$, and $L_1=\Id$, and $r_1=0$. 
\end{enumerate}
The restrictions imposed in A) are quite stringent since many
problems involve at least two nondifferentiable potentials.
Let us also observe that since, in general, there is no explicit 
expression for $\prox_{g_i\circ L_i}$ in terms of 
$\prox_{g_i}$ and $L_i$, Problem~\ref{prob:1} cannot be 
reduced to the setting described in B). On the other hand, 
using a product space reformulation, we shall show that the setting 
described in C) can be exploited to solve Problem~\ref{prob:1}
using only approximate implementations of the operators
$(\prox_{g_i})_{1\leq i\leq m}$.
Our algorithm is introduced in Section~\ref{sec:2}, where we also
establish its convergence properties.
In Section~\ref{sec:3}, our results are applied to best 
approximation and image recovery problems. 

Our notation is standard.
$\BL(\HH,\GG)$ is the space of bounded linear operators from $\HH$ 
to a real Hilbert space $\GG$. The adjoint of $L\in\BL(\HH,\GG)$ is 
denoted by $L^*$. The conjugate of $f\in\Gamma_0(\HH)$ 
is the function $f^*\in\Gamma_0(\HH)$ defined by 
$f^*\colon u\mapsto\sup_{x\in\HH}(\scal{x}{u}-f(x))$. The
projector onto a nonempty closed convex set $C\subset\HH$ 
is denoted by $P_C$. The strong relative interior of a 
convex set $C\subset\HH$ is 
\begin{multline}
\label{e:sri}
\sri C=\menge{x\in C}{\cone (C-x)=\spc (C-x)},\\
\quad\text{where}\quad
\cone C=\bigcup_{\lambda>0}\menge{\lambda x}{x\in C},
\end{multline}
and the relative interior of $C$ is 
$\reli C=\menge{x\in C}{\cone (C-x)=\spa (C-x)}$.
We have $\inte C\subset\sri C\subset\reli C\subset C$ and, if $\HH$ is 
finite-dimensional, $\reli C=\sri C$.
For background on convex analysis, see \cite{Zali02}.

\section{Main result}
\label{sec:2}

To solve Problem~\ref{prob:1}, we propose the following algorithm.
Its main features are that each function $g_i$ is activated 
individually by means of its proximity operator, and that the 
proximity operators can be evaluated simultaneously. It is important
to stress that the functions $(g_i)_{1\leq i\leq m}$ and the operators
$(L_i)_{1\leq i\leq m}$ are used at separate steps in the algorithm,
which is thus fully decomposed. In addition, an error $a_{i,n}$ is 
tolerated in the evaluation of the $i$th proximity operator at 
iteration $n$. 

\begin{algorithm}
\label{algo:2}
For every $i\in\{1,\ldots,m\}$, let $(a_{i,n})_{n\in\NN}$ be a sequence 
in $\GG_i$.
\begin{equation}
\label{e:main2}
\begin{array}{l}
\operatorname{Initialization}\\
\left\lfloor
\begin{array}{l}
\rho=\big(\max_{1\leq i\leq m}\|L_i\|\big)^{-2}\\[1mm]
\varepsilon\in\left]0,\min\{1,\rho\}\right[\\[1mm]
\operatorname{For}\;i=1,\ldots,m\\
\left\lfloor
\begin{array}{l}
v_{i,0}\in\GG_i\\
\end{array}
\right.\\[1mm]
\end{array}
\right.\\[10mm]
\operatorname{For}\;n=0,1,\ldots\\
\left\lfloor
\begin{array}{l}
x_n=z-\sum_{i=1}^m\omega_iL_i^*v_{i,n}\\[1mm]
\gamma_n\in\left[\varepsilon,2\rho-\varepsilon\right]\\[1mm]
\lambda_n\in\left[\varepsilon,1\right]\\
\operatorname{For}\;i=1,\ldots,m\\
\left\lfloor
\begin{array}{l}
v_{i,n+1}=v_{i,n}+\lambda_n\Big(\prox_{\gamma_n g_i^*}
\big(v_{i,n}+\gamma_n(L_ix_n-r_i)\big)+a_{i,n}-v_{i,n}\Big).
\end{array}
\right.\\[2mm]
\end{array}
\right.\\[2mm]
\end{array}
\end{equation}
\end{algorithm}

Note that an alternative implementation of \eqref{e:main2}
can be obtained via Moreau's decomposition formula in a real
Hilbert space $\GG$ \cite[Lemma~2.10]{Smms05}
\begin{equation}
\label{e:2010-06-15b}
(\forall g\in\Gamma_0(\GG))(\forall\gamma\in\RPP)(\forall v\in\GG)
\quad\prox_{\gamma g^*}v=v-\gamma\prox_{\gamma^{-1}g}(\gamma^{-1}v).
\end{equation}
We now describe the asymptotic behavior of Algorithm~\ref{algo:2}.

\begin{theorem}
\label{t:2}
Suppose that
\begin{equation}
\label{e:hanoi2010-05-27a}
(r_i)_{1\leq i\leq m}\in\sri
\menge{(L_ix-y_i)_{1\leq i\leq m}}{x\in\HH, 
(y_i)_{1\leq i\leq m}\in\cart_{i=1}^m\dom g_i}
\end{equation}
and that
\begin{equation}
\label{e:2010-06-15a}
(\forall i\in\{1,\ldots,m\})\quad\sum_{n\in\NN}\|a_{i,n}\|_{\GG_i}<\pinf. 
\end{equation}
Furthermore, let $(x_n)_{n\in\NN}$, $(v_{1,n})_{n\in\NN}$, \ldots,
$(v_{m,n})_{n\in\NN}$ be sequences generated by 
Algorithm~\ref{algo:2}.
Then Problem~\ref{prob:1} possesses a unique solution $x$ and
the following hold.
\begin{enumerate}
\item
\label{t:2i}
For every $i\in\{1,\ldots,m\}$, $(v_{i,n})_{n\in\NN}$ converges 
weakly to a point $v_i\in\GG_i$. Moreover, $(v_i)_{1\leq i\leq m}$ 
is a solution to the minimization problem 
\begin{equation}
\label{e:prob2}
\underset{v_1\in\GG_1,\ldots,\, v_m\in\GG_m}
{\mathrm{minimize}}\;\;\frac12\left\|z-\sum_{i=1}^m
\omega_iL_i^*v_i\right\|^2+\sum_{i=1}^m\omega_i
\big(g_i^*(v_i)+\scal{v_i}{r_i}\big),
\end{equation}
and $x=z-\sum_{i=1}^m\omega_iL_i^*v_i$.
\item
\label{t:2ii}
$(x_n)_{n\in\NN}$ converges strongly to $x$.
\end{enumerate}
\end{theorem}
\begin{proof}
Set $f\colon\HH\to\RX\colon x\mapsto\sum_{i=1}^m\omega_ig_i(L_ix-r_i)$.
The assumptions imply that, for every $i\in\{1,\ldots,m\}$, the 
function $x\mapsto g_i(L_ix-r_i)$ is convex and lower semicontinuous.
Hence, $f$ is likewise. On the other hand, it follows from 
\eqref{e:hanoi2010-05-27a} that 
\begin{equation}
(r_i)_{1\leq i\leq m}\in\menge{(L_ix-y_i)_{1\leq i\leq m}}{x\in\HH, 
(y_i)_{1\leq i\leq m}\in\cart_{i=1}^m\dom g_i} 
\end{equation}
and, therefore, that
$\dom f\neq\emp$. Thus, $f\in\Gamma_0(\HH)$ and, as seen in 
\eqref{e:prox}, Problem~\ref{prob:1} possesses a unique 
solution, namely $x=\prox_f z$.

Now let $\HHH$ be the real Hilbert space obtained by endowing 
the Cartesian product $\HH^m$ with the scalar product 
$\scal{\cdot}{\cdot}_{\HHH}\colon(\boldsymbol{x},\boldsymbol{y})\mapsto
\sum_{i=1}^m\omega_i\scal{x_i}{y_i}$, where 
$\boldsymbol{x}=(x_i)_{1\leq i\leq m}$ and 
$\boldsymbol{y}=(y_i)_{1\leq i\leq m}$ denote generic 
elements in $\HHH$. The associated norm is 
\begin{equation}
\label{e:palawan-mai2008}
\|\cdot\|_{\HHH}\colon\boldsymbol{x}\mapsto
\sqrt{\sum_{i=1}^m\omega_i\|x_i\|^2}.
\end{equation}
Likewise, let $\GGG$ denote the real Hilbert space obtained by 
endowing the Cartesian product $\GG_1\times\cdots\times\GG_m$ 
with the scalar product and the associated norm respectively 
defined by
\begin{equation}
\label{e:palawan-mai2008-}
\scal{\cdot}{\cdot}_{\GGG}
\colon(\boldsymbol{y},\boldsymbol{z})\mapsto
\sum_{i=1}^m\omega_i\scal{y_i}{z_i}_{\GG_i}
\quad\text{and}\quad\|\cdot\|_{\GGG}\colon
\boldsymbol{y}\mapsto\sqrt{\sum_{i=1}^m\omega_i\|y_i\|_{\GG_i}^2}.
\end{equation}
Define 
\begin{equation}
\label{e:fgL}
\begin{cases}
{\boldsymbol f}=\iota_{\boldsymbol D},
\quad\text{where}\quad\boldsymbol{D}=\menge{(x,\dots,x)\in\HHH}{x\in\HH}\\
{\boldsymbol g}\colon\GGG\to\RX\colon\boldsymbol{y}\mapsto 
\sum_{i=1}^m\omega_ig_i(y_i)\\
{\boldsymbol L}\colon\HHH\to\GGG\colon\boldsymbol{x}\mapsto
(L_ix_i)_{1\leq i\leq m}\\
\boldsymbol{r}=(r_1,\ldots,r_m)\\
\boldsymbol{z}=(z,\ldots,z).
\end{cases}
\end{equation}
Then $\boldsymbol{f}\in\Gamma_0(\HHH)$, 
$\boldsymbol{g}\in\Gamma_0(\GGG)$, and 
$\boldsymbol{L}\in\BL(\HHH,\GGG)$. Moreover, $\boldsymbol{D}$ 
is a closed vector subspace of $\HHH$ with projector
\begin{equation}
\label{e:PD}
\prox_{\boldsymbol{f}}=
P_{\boldsymbol{D}}\colon\boldsymbol{x}\mapsto
\bigg(\sum_{i=1}^m\omega_ix_i,\ldots,\sum_{i=1}^m\omega_ix_i\bigg)
\end{equation}
and  
\begin{equation}
\label{e:L*}
\boldsymbol{L}^*\colon\GGG\to\HHH\colon\boldsymbol{v}\mapsto
\big(L_i^*v_i\big)_{1\leq i\leq m}.
\end{equation}
Note that \eqref{e:palawan-mai2008-} and \eqref{e:palawan-mai2008} 
yield
\begin{align}
\label{e:2010-06-09e}
(\forall\boldsymbol{x}\in\HHH)\quad
\|\boldsymbol{L}\boldsymbol{x}\|_{\GGG}^2
&=\sum_{i=1}^m\omega_i\|L_ix_i\|_{\GG_i}^2\nonumber\\
&\leq\sum_{i=1}^m\omega_i\|L_i\|^2\|x_i\|^2\nonumber\\
&\leq\Big(\max_{1\leq i\leq m}\|L_i\|^2\Big)
\sum_{i=1}^m\omega_i\|x_i\|^2\nonumber\\
&=\Big(\max_{1\leq i\leq m}\|L_i\|^2\Big)\|\boldsymbol{x}\|^2_{\HHH}.
\end{align}
Therefore, 
\begin{equation}
\label{e:2010-06-09f}
\|\boldsymbol{L}\|\leq\max_{1\leq i\leq m}\|L_i\|.
\end{equation}
We also deduce from \eqref{e:hanoi2010-05-27a} that 
\begin{equation}
\label{e:1938}
\boldsymbol{r}\in\sri\big(\boldsymbol{L}(\dom \boldsymbol{f})-
\dom\boldsymbol{g}\big).
\end{equation}
Furthermore, in view of \eqref{e:palawan-mai2008} and
\eqref{e:fgL}, in the space $\HHH$, 
\eqref{e:prob1} is equivalent to
\begin{equation}
\label{e:prob4}
\underset{\boldsymbol{x}\in\HHH}{\mathrm{minimize}}\;\;
{\boldsymbol f}(\boldsymbol{x})+
{\boldsymbol g}({\boldsymbol L}\boldsymbol{x}-\boldsymbol{r})+
\frac12\|\boldsymbol{x}-\boldsymbol{z}\|_{\HHH}^2.
\end{equation}
Next, we derive from \cite[Proposition~3.3]{Svva10} that the dual 
problem of \eqref{e:prob4} is to
\begin{equation}
\label{e:prob5}
\underset{\boldsymbol{v}\in\GGG}{\mathrm{minimize}}\;\;
\widetilde{\boldsymbol{f}^*}(\boldsymbol{z}-
\boldsymbol{L}^*\boldsymbol{v})+\boldsymbol{g}^*(\boldsymbol{v})
+\scal{\boldsymbol{v}}{\boldsymbol{r}}_{\GGG},
\end{equation}
where $\widetilde{\boldsymbol{f}^*}\colon\boldsymbol{u}\mapsto
\inf_{\boldsymbol{w}\in\HHH}\big(\boldsymbol{f}^*(\boldsymbol{w})+
(1/2)\|\boldsymbol{u}-\boldsymbol{w}\|_{\HHH}^2\big)$ is 
the Moreau envelope of $\boldsymbol{f}^*$. 
Since ${\boldsymbol f}=\iota_{\boldsymbol D}$, we have
$\boldsymbol{f}^*=\iota_{\boldsymbol{D}^\bot}$.
Hence, \eqref{e:palawan-mai2008} and \eqref{e:PD} yield
\begin{equation}
\label{e:2010-06-09a}
(\forall\boldsymbol{u}\in\HHH)\quad
\widetilde{\boldsymbol{f}^*}(\boldsymbol{u})
=\frac12\|\boldsymbol{u}-P_{\boldsymbol{D}^\bot}\boldsymbol{u}\|_{\HHH}^2
=\frac12\|P_{\boldsymbol{D}}\boldsymbol{u}\|_{\HHH}^2
=\frac12\left\|\sum_{i=1}^m\omega_iu_i\right\|^2.
\end{equation}
On the other hand, \eqref{e:palawan-mai2008-} and \eqref{e:fgL} yield
\begin{equation}
\label{e:2010-06-09b}
(\forall\boldsymbol{v}\in\GGG)\quad
\boldsymbol{g}^*(\boldsymbol{v})=
\sum_{i=1}^m\omega_ig_i^*(v_i)\quad\text{and}\quad
\prox_{\boldsymbol{g}^*}\boldsymbol{v}
=\big(\prox_{g_i^*}v_i\big)_{1\leq i\leq m}.
\end{equation}
Altogether, it follows from \eqref{e:L*}, \eqref{e:2010-06-09a}, 
\eqref{e:2010-06-09b}, and \eqref{e:palawan-mai2008-}, that 
\begin{equation}
\label{e:2010-06-09c}
\text{\eqref{e:prob5} is equivalent to \eqref{e:prob2}.}
\end{equation}
Now define
\begin{equation}
\label{e:2010-06-09d}
(\forall n\in\NN)\quad
\begin{cases}
\boldsymbol{x}_n=(x_n,\ldots,x_n)\\
\boldsymbol{v}_n=(v_{1,n},\ldots,v_{m,n})\\
\boldsymbol{a}_n=(a_{1,n},\ldots,a_{m,n}).
\end{cases}
\end{equation}
Then, in view of \eqref{e:fgL}, \eqref{e:PD}, \eqref{e:L*}, 
\eqref{e:2010-06-09f}, and \eqref{e:2010-06-09b}, \eqref{e:main2} 
is a special case of the following routine.
\begin{equation}
\label{e:main1}
\begin{array}{l}
\operatorname{Initialization}\\
\left\lfloor
\begin{array}{l}
\rho=\|{\boldsymbol L}\|^{-2}\\[1mm]
\varepsilon\in\left]0,\min\{1,\rho\}\right[\\[1mm]
{\boldsymbol v}_0\in\GGG\\[1mm]
\end{array}
\right.\\[5mm]
\operatorname{For}\;n=0,1,\ldots\\
\left\lfloor
\begin{array}{l}
{\boldsymbol x}_n=\prox_{\boldsymbol f}({\boldsymbol z}-
{\boldsymbol L}^*{\boldsymbol v}_n)\\[1mm]
\gamma_n\in\left[\varepsilon,2\rho-\varepsilon\right]\\[1mm]
\lambda_n\in\left[\varepsilon,1\right]\\
{\boldsymbol v}_{n+1}={\boldsymbol v}_n
+\lambda_n\big(\prox_{\gamma_n {\boldsymbol g}^*}({\boldsymbol v}_n
+\gamma_n({\boldsymbol L}{\boldsymbol x}_n
-{\boldsymbol r}))+{\boldsymbol a}_n-{\boldsymbol v}_n\big).
\end{array}
\right.\\[2mm]
\end{array}
\end{equation}
Moreover, \eqref{e:2010-06-15a} implies that
$\sum_{n\in\NN}\|{\boldsymbol a}_n\|_{\GGG}<\pinf$. 
Hence, it follows from \eqref{e:1938} and
\cite[Theorem~3.7]{Svva10} that the following hold, 
where $\boldsymbol{x}$ is the solution to 
\eqref{e:prob4}.
\begin{itemize}
\item[(a)]
$({\boldsymbol v}_n)_{n\in\NN}$ converges weakly to a solution
$\boldsymbol{v}$ to \eqref{e:prob5} and ${\boldsymbol x}=
\prox_{\boldsymbol f}({\boldsymbol z}-{\boldsymbol L}^*{\boldsymbol v})$.
\item[(b)]
$({\boldsymbol x}_n)_{n\in\NN}$ converges strongly to
${\boldsymbol x}$.
\end{itemize}
In view of \eqref{e:palawan-mai2008}, \eqref{e:palawan-mai2008-},
\eqref{e:fgL}, \eqref{e:PD}, \eqref{e:L*}, \eqref{e:2010-06-09c}, 
and \eqref{e:2010-06-09d}, items (a) and (b) provide respectively 
items \ref{t:2i} and \ref{t:2ii}. 
\end{proof}

\begin{remark}
\label{r:2010-06-18}
Let us consider Problem~\ref{prob:1} in the special case when
$(\forall i\in\{1,\ldots,m\})$ $\GG_i=\HH$, $L_i=\Id$, and $r_i=0$.
Then \eqref{e:prob1} reduces to 
\begin{equation}
\label{e:prob24}
\underset{x\in\HH}{\mathrm{minimize}}\;\;
\sum_{i=1}^m\omega_ig_i(x)+\frac12\|x-z\|^2.
\end{equation}
Now let us implement Algorithm~\ref{algo:2} with 
$\gamma_n\equiv 1$, $\lambda_n\equiv 1$, 
$a_{i,n}\equiv 0$, and $v_{i,0}\equiv 0$.
The iteration process resulting from \eqref{e:main2} can be 
written as 
\begin{equation}
\label{e:main23}
\begin{array}{l}
\operatorname{Initialization}\\
\left\lfloor
\begin{array}{l}
x_0=z\\
\operatorname{For}\;i=1,\ldots,m\\
\left\lfloor
\begin{array}{l}
v_{i,0}=0\\
\end{array}
\right.\\[1mm]
\end{array}
\right.\\[4mm]
\operatorname{For}\;n=0,1,\ldots\\
\left\lfloor
\begin{array}{l}
\operatorname{For}\;i=1,\ldots,m\\
\left\lfloor
\begin{array}{l}
v_{i,n+1}=\prox_{g^*_i}(x_n+v_{i,n}).
\end{array}
\right.\\[1mm]
x_{n+1}=z-\sum_{i=1}^m\omega_iv_{i,n+1}.
\end{array}
\right.\\[2mm]
\end{array}
\end{equation}
For every $i\in\{1,\ldots,m\}$ and $n\in\NN$, set
$z_{i,n}=x_n+v_{i,n}$. Then \eqref{e:main23} yields
\begin{equation}
\label{e:main24}
\begin{array}{l}
\operatorname{Initialization}\\
\left\lfloor
\begin{array}{l}
x_0=z\\
\operatorname{For}\;i=1,\ldots,m\\
\left\lfloor
\begin{array}{l}
z_{i,0}=z\\
\end{array}
\right.\\[1mm]
\end{array}
\right.\\[5mm]
\operatorname{For}\;n=0,1,\ldots\\
\left\lfloor
\begin{array}{l}
x_{n+1}=z-\sum_{i=1}^m\omega_i\prox_{g^*_i}z_{i,n}\\
\operatorname{For}\;i=1,\ldots,m\\
\left\lfloor
\begin{array}{l}
z_{i,n+1}=x_{n+1}+\prox_{g^*_i}z_{i,n}.
\end{array}
\right.\\[0mm]
\end{array}
\right.\\[6mm]
\end{array}
\end{equation}
Next we observe that
$(\forall n\in\NN)$ $\sum_{i=1}^m\omega_iz_{i,n}=z$. 
Indeed, the identity is clearly satisfied for $n=0$ and,
for every $n\in\NN$, \eqref{e:main24} yields
$\sum_{i=1}^m\omega_iz_{i,n+1}=x_{n+1}+\sum_{i=1}^m\omega_i
\prox_{g^*_i}z_{i,n}=(z-\sum_{i=1}^m\omega_i\prox_{g^*_i}z_{i,n})
+\sum_{i=1}^m\omega_i\prox_{g^*_i}z_{i,n}=z$.
Thus, invoking \eqref{e:2010-06-15b} with $\gamma=1$,  
we can rewrite \eqref{e:main24} as
\begin{equation}
\label{e:main25}
\begin{array}{l}
\operatorname{Initialization}\\
\left\lfloor
\begin{array}{l}
x_0=z\\
\operatorname{For}\;i=1,\ldots,m\\
\left\lfloor
\begin{array}{l}
z_{i,0}=z\\
\end{array}
\right.\\[1mm]
\end{array}
\right.\\[5mm]
\operatorname{For}\;n=0,1,\ldots\\
\left\lfloor
\begin{array}{l}
x_{n+1}=\sum_{i=1}^m\omega_i\prox_{g_i}z_{i,n}\\
\operatorname{For}\;i=1,\ldots,m\\
\left\lfloor
\begin{array}{l}
z_{i,n+1}=x_{n+1}+z_{i,n}-\prox_{g_i}z_{i,n}.
\end{array}
\right.\\[0mm]
\end{array}
\right.\\[6mm]
\end{array}
\end{equation}
This is precisely the Dykstra-like algorithm proposed in 
\cite[Theorem~4.2]{Joca09} for computing
$\prox_{\sum_{i=1}^m\omega_ig_i}z$ (which itself extends the
classical parallel Dykstra algorithm for projecting $z$ onto an 
intersection of closed convex sets \cite{Baus94,Gaff89}). Hence,
Algorithm~\ref{algo:2} can be viewed as an extension of
this algorithm, which was derived and analyzed with different
techniques in \cite{Joca09}.
\end{remark}

\section{Applications}
\label{sec:3}

As noted in the Introduction, special cases of Problem~\ref{prob:1} 
have already been considered in the literature under certain
restrictions on the number $m$ of composite functions, the 
complexity of the linear operators $(L_i)_{1\leq i\leq m}$, 
and/or the smoothness of the potentials $(g_i)_{1\leq i\leq m}$
(one will find specific applications in 
\cite{Cham05,Svva10,Banf10,Smms05,Demo09,Pott93} 
and the references therein). The proposed framework makes it 
possible to remove these restrictions simultaneously. In this 
section, we provide two illustrations.

\subsection{Best approximation from an intersection of composite 
convex sets}

In this section, we consider the problem of finding the best
approximation $P_Dz$ to a point $z\in\HH$ from a closed convex 
subset $D$ of $\HH$ defined as an intersection of affine inverse 
images of closed convex sets. 

\begin{problem}
\label{prob:3}
Let $z\in\HH$ and, for every $i\in\{1,\ldots,m\}$, let 
$(\GG_i,\|\cdot\|_{\GG_i})$ be a real Hilbert space, let 
$r_i\in\GG_i$, let $C_i$ be a nonempty closed convex subset of
$\GG_i$, and let $0\neq L_i\in\BL(\HH,\GG_i)$. The problem is to 
\begin{equation}
\label{e:prob3}
\underset{x\in D}{\mathrm{minimize}}\;\|x-z\|,
\quad\text{where}\quad
D=\bigcap_{i=1}^m\menge{x\in\HH}{L_ix\in r_i+C_i}.
\end{equation}
\end{problem}

In view of \eqref{e:iota}, Problem~\ref{prob:3} is a special case of 
Problem~\ref{prob:1}, where $(\forall i\in\{1,\ldots,m\})$
$g_i=\iota_{C_i}$ and $\omega_i=1/m$. 
It follows that, for every
$i\in\{1,\ldots,m\}$ and every $\gamma\in\RPP$, $\prox_{\gamma g_i}$ 
reduces to the projector $P_{C_i}$ onto $C_i$. Hence, using
\eqref{e:2010-06-15b}, we can rewrite Algorithm~\ref{algo:2} 
in the following form, where we have set 
$c_{i,n}=-\gamma_n^{-1}a_{i,n}$ for simplicity.

\begin{algorithm}
\label{algo:3}
For every $i\in\{1,\ldots,m\}$, let $(c_{i,n})_{n\in\NN}$ be a sequence 
in $\GG_i$.
\begin{equation}
\label{e:main3}
\begin{array}{l}
\operatorname{Initialization}\\
\left\lfloor
\begin{array}{l}
\rho=\big(\max_{1\leq i\leq m}\|L_i\|\big)^{-2}\\[1mm]
\varepsilon\in\left]0,\min\{1,\rho\}\right[\\[1mm]
\operatorname{For}\;i=1,\ldots,m\\
\left\lfloor
\begin{array}{l}
v_{i,0}\in\GG_i\\
\end{array}
\right.\\[1mm]
\end{array}
\right.\\[10mm]
\operatorname{For}\;n=0,1,\ldots\\
\left\lfloor
\begin{array}{l}
x_n=z-\sum_{i=1}^m\omega_iL_i^*v_{i,n}\\[1mm]
\gamma_n\in\left[\varepsilon,2\rho-\varepsilon\right]\\[1mm]
\lambda_n\in\left[\varepsilon,1\right]\\
\operatorname{For}\;i=1,\ldots,m\\
\left\lfloor
\begin{array}{l}
v_{i,n+1}=v_{i,n}+\gamma_n\lambda_n\Big(
L_ix_n-r_i-
P_{C_i}\big(\gamma_n^{-1}v_{i,n}+L_ix_n-r_i\big)-c_{i,n}\Big).
\end{array}
\right.\\[2mm]
\end{array}
\right.\\[2mm]
\end{array}
\end{equation}
\end{algorithm}

In the light of the above, we obtain the following application of
Theorem~\ref{t:2}\ref{t:2ii}.

\begin{corollary}
\label{c:2010-07-07a}
Suppose that
\begin{equation}
\label{e:2010-07-07a}
(r_i)_{1\leq i\leq m}\in\sri
\menge{(L_ix-y_i)_{1\leq i\leq m}}{x\in\HH, 
(y_i)_{1\leq i\leq m}\in\cart_{i=1}^m C_i}
\end{equation}
and that $(\forall i\in\{1,\ldots,m\})$ 
$\sum_{n\in\NN}\|c_{i,n}\|_{\GG_i}<\pinf$.
Then every sequence $(x_n)_{n\in\NN}$ generated by 
Algorithm~\ref{algo:3} converges strongly to the 
solution $P_Dz$ to Problem~\ref{prob:3}.
\end{corollary}

\subsection{Nonsmooth image recovery}

A wide range of signal and image recovery problems can be modeled 
as instances of Problem~\ref{prob:1}. In this section, we focus on 
the problem of recovering an image $\overline{x}\in\HH$ from 
$p$ noisy measurements
\begin{equation}
\label{e:model}
r_i=T_i\overline{x}+s_i,\quad 1\leq i\leq p.
\end{equation}
In this model, the $i$th measurement $r_i$ lies in a Hilbert space
$\GG_i$, $T_i\in\BL(\HH,\GG_i)$ is the data formation operator,
and $s_i\in\GG_i$ is the realization of a noise process. A typical
data fitting potential in such models is the function 
\begin{equation}
\label{e:2010-07-14s}
x\mapsto\sum_{i=1}^p\omega_ig_i(T_ix-r_i),\quad\text{where}\quad
0\leq g_i\in\Gamma_0(\GG_i)\quad\text{and $g_i$ vanishes only at $0$}.
\end{equation}
The proposed framework can handle $p\geq 1$ nondifferentiable functions
$(g_i)_{1\leq i\leq p}$ as well as the incorporation of additional 
potential functions to model prior knowledge on the original image 
$\overline{x}$. In the illustration we provide below, 
the following is assumed.
\begin{itemize}
\item
The image space is $\HH={\mathrm H}_0^1(\Omega)$, where 
$\Omega$ is a nonempty bounded open domain in $\RR^2$. 
\item
$\overline{x}$ admits a sparse decomposition in an orthonormal
basis $(e_k)_{k\in\NN}$ of $\HH$. 
As discussed in \cite{Demo09,Zouh05} this property can be promoted 
by the ``elastic net'' potential 
$x\mapsto\sum_{k\in\NN}\phi_k(\scal{x}{e_k})$, where
$(\forall k\in\NN)$
$\phi_k\colon\xi\mapsto\alpha|\xi|+\beta|\xi|^2$, with
$\alpha>0$ and $\beta>0$. More general choices of suitable functions
$(\phi_k)_{k\in\NN}$ are available \cite{Siop07}.
\item
$\overline{x}$ is piecewise smooth. This property is promoted 
by the total variation potential 
$\operatorname{tv}(x)=\int_\Omega|\nabla x(\omega)|_{2}d\omega$,
where $|\cdot|_2$ denotes the Euclidean norm on $\RR^2$ 
\cite{Rudi92}.
\end{itemize}

Upon setting $g_i\equiv\|\cdot\|_{\GG_i}$ in \eqref{e:2010-07-14s}, 
these considerations lead us to the following formulation (see
\cite[Example~2.10]{Svva10} for more general nonsmooth potentials). 

\begin{problem}
\label{prob:9}
Let $\HH={\mathrm H}_0^1(\Omega)$, where $\Omega\subset\RR^2$ is
nonempty, bounded, and open, let $(\omega_i)_{1\leq i\leq p+2}$ be 
reals in $\left]0,1\right]$ such that $\sum_{i=1}^{p+2}\omega_i=1$, 
and let $(e_k)_{k\in\NN}$ be an orthonormal basis of $\HH$.
For every $i\in\{1,\ldots,p\}$, let $0\neq T_i\in\BL(\HH,\GG_i)$,
where $(\GG_i,\|\cdot\|_{\GG_i})$ is a real Hilbert space, 
and let $r_i\in\GG_i$. The problem is to 
\begin{equation}
\label{e:prob11}
\underset{x\in\HH}{\mathrm{minimize}}\;\;
\sum_{i=1}^p\omega_i\|T_ix-r_i\|_{\GG_i}+
\sum_{k\in\NN}\bigg(\omega_{p+1}|\scal{x}{e_k}|
+\frac12|\scal{x}{e_k}|^2\bigg)
+\omega_{p+2}\operatorname{tv}(x).
\end{equation}
\end{problem}

It follows from Parseval's identity that Problem~\ref{prob:9} 
is a special case of Problem~\ref{prob:1} in 
$\HH={\mathrm H}_0^1(\Omega)$ with $m=p+2$, $z=0$, and 
\begin{equation}
\label{e:2010-07-14f}
\begin{cases}
g_i=\|\cdot\|_{\GG_i}\;\text{and}\;L_i=T_i,\:
\text{if}\;\;1\leq i\leq p;\\
\GG_{p+1}=\ell^2(\NN),\,g_{p+1}=\|\cdot\|_{\ell^1},\,
r_{p+1}=0,\,\text{and}\;L_{p+1}\colon x\mapsto(\scal{x}{e_k})_{k\in\NN};\\
\GG_{p+2}={\mathrm L}^2(\Omega)\oplus {\mathrm L}^2(\Omega),\,
g_{p+2}\colon y\mapsto\int_\Omega|y(\omega)|_2d\omega,\,
r_{p+2}=0,\,\text{and}\;L_{p+2}=\nabla.
\end{cases}
\end{equation}

To implement Algorithm~\ref{algo:2}, it suffices to note that 
$L_{p+1}^*\colon(\nu_k)_{k\in\NN}\mapsto\sum_{k\in\NN}\nu_ke_k$ and
$L_{p+2}^*=-\operatorname{div}$, and to specify the proximity 
operators of the functions $(\gamma g^*_i)_{1\leq i\leq m}$,
where $\gamma\in\RPP$.
First, let $i\in\{1,\ldots,p\}$. Then $g_i=\|\cdot\|_{\GG_i}$
and therefore $g_i^*=\iota_{B_i}$, where $B_i$ is the closed unit
ball of $\GG_i$. Hence $\prox_{\gamma g_i^*}=P_{B_i}$.
Next, it follows from \eqref{e:2010-06-15b} and
\cite[Example~2.20]{Smms05} that 
$\prox_{\gamma g^*_{p+1}}\colon(\xi_k)_{k\in\NN}\mapsto
(P_{[-1,1]}\xi_k)_{k\in\NN}$.
Finally, since $g_{p+2}$ is the support function of 
the set \cite{Merc80}
\begin{equation}
\label{e:K}
K=\menge{y\in\GG_{p+2}}{|y|_2\leq 1~\text{a.e.}},
\end{equation}
$g_{p+2}^*=\iota_K$ and therefore $\prox_{\gamma g_{p+2}^*}=P_K$,
which is straightforward to compute. 
Altogether, as $\|L_{p+1}\|=1$ and $\|L_{p+2}\|\leq 1$, 
Algorithm~\ref{algo:2} assumes the following form
(since all the proximity operators can be implemented with
simple projections, we dispense with the errors terms).

\begin{algorithm}
\label{algo:4}
\begin{equation}
\label{e:main9}
\begin{array}{l}
\operatorname{Initialization}\\
\left\lfloor
\begin{array}{l}
\rho=\big(\max\{1,\|T_1\|,\ldots,\|T_p\|\}\big)^{-2}\\[1mm]
\varepsilon\in\left]0,\min\{1,\rho\}\right[\\[1mm]
\operatorname{For}\;i=1,\ldots,p\\
\left\lfloor
\begin{array}{l}
v_{i,0}\in\GG_i\\
\end{array}
\right.\\[1mm]
v_{p+1,0}=(\nu_{k,0})_{k\in\NN}\in\ell^2(\NN)\\
v_{p+2,0}\in{\mathrm L}^2(\Omega)\oplus{\mathrm L}^2(\Omega)\\ 
\end{array}
\right.\\[12mm]
\operatorname{For}\;n=0,1,\ldots\\
\left\lfloor
\begin{array}{l}
x_n=z-\sum_{i=1}^p\omega_iT_i^*v_{i,n}
-\omega_{p+1}\sum_{k\in\NN}\nu_{k,n}e_k
+\omega_{p+2}\operatorname{div}v_{p+2,n}\\[1mm]
\gamma_n\in\left[\varepsilon,2\rho-\varepsilon\right]\\[1mm]
\lambda_n\in\left[\varepsilon,1\right]\\
\operatorname{For}\;i=1,\ldots,p\\
\left\lfloor
\begin{array}{l}
v_{i,n+1}=v_{i,n}+\lambda_n\Big(
{\displaystyle\frac{v_{i,n}+\gamma_n(T_ix_n-r_i)}
{\max\{1,\|v_{i,n}+\gamma_n(T_ix_n-r_i)\|_{\GG_i}\}}}-v_{i,n}\Big)
\end{array}
\right.\\[4mm]
\operatorname{For~every}\;k\in\NN,\:
\nu_{k,n+1}=\nu_{k,n}+\lambda_n\Big(
{\displaystyle\frac{\nu_{k,n}+\gamma_n\scal{x_n}{e_k}}
{\max\{1,|\nu_{k,n}+\gamma_n\scal{x_n}{e_k}|\}}}-\nu_{k,n}\Big)\\[2mm]
\operatorname{For~almost~every}\;\omega\in\Omega,\\
\qquad v_{p+2,n+1}(\omega)=v_{p+2,n}(\omega)+\lambda_n\Big(
{\displaystyle\frac{v_{p+2,n}(\omega)+\gamma_n\nabla x_n(\omega)}
{\max\{1,|v_{p+2,n}(\omega)+\gamma_n\nabla x_n(\omega)|_2\}}}
-v_{p+2,n}(\omega)\Big).
\end{array}
\right.\\[2mm]
\end{array}
\end{equation}
\end{algorithm}

Let us establish the main convergence property of this algorithm.
\begin{corollary}
\label{c:2010-07-21}
Every sequence $(x_n)_{n\in\NN}$ generated by 
Algorithm~\ref{algo:4} converges strongly to the 
solution to Problem~\ref{prob:9}.
\end{corollary}
\begin{proof}
In view of the above discussion and of Theorem~\ref{t:2}\ref{t:2ii}, 
it remains to check that \eqref{e:hanoi2010-05-27a} is satisfied.
Set $S=\menge{(L_ix-y_i)_{1\leq i\leq m}}{x\in\HH, 
(y_i)_{1\leq i\leq m}\in\cart_{i=1}^m\dom g_i}$. 
We have $\dom g_i=\GG_i$ for every $i\in\{1,\ldots,p\}$,
$\dom g_{p+1}=\ell^1(\NN)$, and
$\dom g_{p+2}={\mathrm L}^2(\Omega)\oplus {\mathrm L}^2(\Omega)$.
Consequently,
\begin{align}
S
&=\Big\{(T_1x-y_1,\ldots,T_px-y_p,(\scal{x}{e_k}-\eta_k)_{k\in\NN},
\nabla x-y_{p+2}~\Big|~\nonumber\\
&\mbox{}\hskip 28mm x\in\HH, 
(y_i)_{1\leq i\leq p}\in\cart_{i=1}^p\GG_i,
(\eta_k)_{k\in\NN}\in\ell^1(\NN),
y_{p+2}\in{\mathrm L}^2(\Omega)\oplus {\mathrm L}^2(\Omega)\Big\}
\nonumber\\
&=\big(\cart_{i=1}^p\GG_i\big)\times\ell^2(\NN)
\times\big({\mathrm L}^2(\Omega)\oplus {\mathrm L}^2(\Omega)\big)
\nonumber\\
&=\cart_{i=1}^m\GG_i.
\end{align}
Hence, we trivially have $(r_1,\ldots,r_p,0,0)\in\sri S$.
\end{proof}

Let us emphasize that a novelty of the above variational framework is
to  perform total variation image recovery in the presence of several 
nondifferentiable composite terms, with guaranteed strong convergence 
to the solution to the problem, and with elementary steps in the form 
of simple projections. 
The finite-dimensional version of the algorithm can easily be obtained 
by discretizing the operators $\nabla$ and $\operatorname{div}$ as
in \cite{Cham05} (see also \cite[Section~4.4]{Svva10} for variants 
of the total variation potential).

\end{document}